\documentclass[11pt]{amsart}

\usepackage{amscd,amssymb,amsopn,amsmath,amsthm,mathrsfs,graphics,amsfonts,enumerate,verbatim,calc
}

\usepackage[all]{xy}

\usepackage{color}

\usepackage[OT2,OT1]{fontenc}
\newcommand\cyr{%
\renewcommand\rmdefault{wncyr}%
\renewcommand\sfdefault{wncyss}%
\renewcommand\encodingdefault{OT2}%
\normalfont
\selectfont}
\DeclareTextFontCommand{\textcyr}{\cyr}

\usepackage{amssymb,amsmath}

\DeclareFontFamily{OT1}{rsfs}{}
\DeclareFontShape{OT1}{rsfs}{n}{it}{<-> rsfs10}{}
\DeclareMathAlphabet{\mathscr}{OT1}{rsfs}{n}{it}

\numberwithin{equation}{section}
\newtheorem{theorem}{Theorem}[section]
\newtheorem{lemma}[theorem]{Lemma}
\newtheorem{proposition}[theorem]{Proposition}
\newtheorem{corollary}[theorem]{Corollary}
\newtheorem{claim}[theorem]{Claim}

\theoremstyle{definition}
\newtheorem{definition}[theorem]{Definition}
\newtheorem{remark}[theorem]{Remark}
\theoremstyle{remark}

\newtheorem{example}[theorem]{Example}

\newcommand{\im}{\operatorname{Im}}
\renewcommand{\ker}{\operatorname{Ker}}

\newcommand{\Spec}{\operatorname{Spec}}

\newcommand{\Ext}{\operatorname{Ext}}

\newcommand{\Hom}{\operatorname{Hom}}

\newcommand{\depth}{\operatorname{depth}}

\newcommand{\coker}{\operatorname{Coker}}

\newcommand{\fm}{\frak{m}}
\newcommand{\fp}{\frak{p}}

\begin{document}
\title[Deformation problem]{Frobenius actions on local cohomology modules and deformation}

\author[Linquan Ma]{Linquan Ma}
\address{Department of Mathematics, University of Utah, Salt Lake City, UT 84102, USA}
\email{lquanma@math.utah.edu}

\author[Pham Hung Quy]{Pham Hung Quy}
\address{Department of Mathematics, FPT University, and Thang Long Institute of Mathematics and Applied Sciences, Ha Noi, Viet Nam}
\email{quyph@fe.edu.vn}

\thanks{
L. Ma is supported in part by the NSF Grant DMS \#1600198 and NSF CAREER Grant DMS \#1252860/1501102, and was partially supported by a Simons Travel Grant when preparing this article. P.H. Quy is partially supported by a fund of Vietnam National Foundation for Science and Technology Development (NAFOSTED) under grant number 101.04-2017.10. This paper was written while the second author was visiting Vietnam Institute for Advanced Study in Mathematics. He would like to thank the VIASM for hospitality and financial support.}


\begin{abstract}
Let $(R, \fm)$ be a Noetherian local ring of characteristic $p>0$. We introduce and study $F$-full and $F$-anti-nilpotent singularities, both are defined in terms of the Frobenius actions on the local cohomology modules of $R$ supported at the maximal ideal. We prove that if $R/(x)$ is $F$-full or $F$-anti-nilpotent for a nonzerodivisor $x\in R$, then so is $R$. We use these results to obtain new cases on the deformation of $F$-injectivity.
\end{abstract}

\maketitle
\section{Introduction}
Let $(R, \fm)$ be a Noetherian local ring of prime characteristic $p>0$. We have the Frobenius endomorphism $F: R \to R, x \mapsto x^p$. The \textit{F-singularities} are certain singularities defined via this Frobenius map. They appear in the theory of \textit{tight closure} (cf. \cite{H96} for its introduction), which was systematically introduced by Hochster and Huneke \cite{HH90} and developed by many researchers, including Hara, Schwede, Smith, Takagi, Watanabe, Yoshida and others. A recent active research of $F$-singularities is centered around the correspondence with the singularities of the minimal model program. We recommend \cite{TW14} as an excellent survey for recent developments.

In this paper we study the deformation of $F$-singularities. That is, we consider the problem: if $R/(x)$ has certain property {\bf P} for a regular element $x\in R$, then does $R$ has the property {\bf P}? The classical objects of $F$-singularities are $F$-regularity, $F$-rationality, $F$-purity and $F$-injectivity (cf. \cite{H96,TW14}). It is well-known that $F$-rationality always deforms while $F$-regularity and $F$-purity do not deform in general \cite{Sin99a, Sin99b}. Whether $F$-injectivity deforms is a long standing open problem \cite{F83} (for recent developments, we refer to \cite{HMS14, MSS16}). Recall that the Frobenius endomorphism induces a natural Frobenius action on every local cohomology module, $F$: $H^i_{\fm}(R) \to H^i_{\fm} (R)$. The ring $R$ is called {\it $F$-injective} if this Frobenius action $F$ is injective for every $i \ge 0$. The class of $F$-injective singularities contains other classes of $F$-singularities. For an ideal-theoretic characterization of $F$-injectivity, see \cite[Main Theorem D]{QS16}. We consider this paper as a step towards a solution of the deformation of $F$-injectivity.

We introduce two conditions: $F$-full and $F$-anti-nilpotent singularities, in terms of the Frobenius actions on local cohomology modules of $R$ (we refer to section 2 for detailed definitions). The first condition is motivated by recent results on Du Bois singularities \cite{MSS16}. The second condition has been studied in \cite{EH08, M14}, and is known to be equivalent to {\it stably FH finite}, which means all local cohomology modules of $R$ and $R[[x_1,\dots,x_n]]$ supported at the maximal ideals have only finitely many Frobenius stable submodules. We prove that $F$-fullness and $F$-anti-nilpotency both deform, and we obtain more evidence on deformation of $F$-injectivity. Our results largely generalize earlier results of \cite{HMS14} in this direction. We list some of our main results here:

\begin{theorem}[Theorem \ref{stably}, Corollary \ref{str deform}]
$(R, \fm)$ be a Noetherian local ring of characteristic $p>0$ and $x$ a regular element of $R$. Then we have:
\begin{enumerate}
\item If $R/(x)$ is $F$-anti-nilpotent, then so is $R$.
\item If $R/(x)$ is $F$-full, then so is $R$.
\item If $R/(x)$ is $F$-full and $F$-injective, then so is $R$.
\end{enumerate}
\end{theorem}

\begin{theorem}[Theorem \ref{stri filt}]
Let $(R, \fm)$ be a Noetherian local ring of characteristic $p>0$. Suppose the residue field $k = R/\fm$ is perfect. Let $x$ be a regular element of $R$ such that $\coker (H^i_{\fm}(R) \overset{x}{\to} H^i_{\fm}(R))$  has finite length for every $i$. If $R/(x)$ is $F$-injective, then the map $x^{p-1}F$: $H_{\fm}^i(R)\to H_{\fm}^i(R)$ is injective for every $i$, in particular $R$ is $F$-injective.
\end{theorem}


\section{Definitions and basic properties}

\subsection{Modules with Frobenius structure}
Let $(R,\fm)$ be a local ring of characteristic $p>0$. A Frobenius action on an $R$-module $M$, $F$: $M\rightarrow M$, is an additive map such that for all $u\in M$ and $r\in R$, $F(ru)=r^pu$. Such an action induces a natural $R$-linear map $\mathscr{F}_R(M)\to M$,\footnote{It is not hard to see that an $R$-linear map $\mathscr{F}_R(M)\to M$ also determines a Frobenius action on $M$.} where $\mathscr{F}_R(-)$ denotes the Peskine-Szpiro's Frobenius functor. We say $N$ is an {\it $F$-stable} submodule of $M$ if $F(N)\subseteq N$. We say the Frobenius action on $M$ is {\it nilpotent} if $F^e(M)=0$ for some $e$.

We note that having a Frobenius action on $M$ is the same as saying that $M$ is a left module over the ring $R\{F\}$, which may be viewed as
a noncommutative ring generated over $R$ by the symbols $1,F,F^2,\dots$ by requiring that $Fr=r^pF$ for $r\in R$. Moreover, $N$ is an $F$-stable submodule of $M$ is equivalent to requiring that $N$ is an $R\{F\}$-submodule of $M$. We will not use this viewpoint in this article though.

Let $M$ be an (typically Artinian) $R$-module with a Frobenius action $F$. We say the Frobenius action on $M$ is {\it full} (or simply $M$ is full), if the map $\mathscr{F}_R^e(M)\to M$ is surjective for some (equivalently, every) $e\geq 1$. This is the same as saying that the $R$-span of all the elements of the form $F^e(u)$ is the whole $M$ for some (equivalently, every) $e\geq 1$. We say the Frobenius action on $M$ is {\it anti-nilpotent} (or simply $M$ is anti-nilpotent), if for any $F$-stable submodule $N\subseteq M$, the induced Frobenius action $F$ on $M/N$ is injective (note that this in particular implies that $F$ acts injectively on $M$).

\begin{lemma}
\label{anti-nilp implies full}
The Frobenius action on $M$ is anti-nilpotent if and only if every $F$-stable submodule $N\subseteq M$ is full. In particular, if $M$ anti-nilpotent, then $M$ is full.
\end{lemma}
\begin{proof}
Suppose $M$ is anti-nilpotent. Let $N\subseteq M$ be an $F$-stable submodule. Consider the $R$-span of $F(N)$, call it $N'$. Clearly, $N'\subseteq N$ is another $F$-stable submodule of $M$ and $F(N)\subseteq N'$. But since $M$ is anti-nilpotent, $F$ acts injectively on $M/N'$. Thus we have $N=N'$ and hence $N$ is full.

Conversely, suppose every $F$-stable submodule of $M$ is full. Suppose there exists an $F$-stable submodule $N\subseteq M$ such that the Frobenius action on $M/N$ is not injective. Pick $y\notin N$ such that $F(y)\in N$. Let $N''=N+Ry$. It is clear that $N''$ is an $F$-stable submodule of $M$ and the $R$-span of $F(N'')$ is contained in $N\subsetneq N''$. This shows $N''$ is not full, a contradiction.
\end{proof}

We also mention that whenever $M$ is endowed with a Frobenius action $F$, then $\widetilde{F}=rF$ defines another Frobenius action on $M$ for every $r\in R$. It is easy to check that if the action $\widetilde{F}$ is full or anti-nilpotent, then so is $F$.

\subsection{$F$-singularities} We collect some definitions about singularities in positive characteristic. Let $(R, \fm)$ be a Noetherian local ring of characteristic $p>0$ with the Frobenius endomorphism $F: R \to R; x \mapsto x^p$. $R$ is called {\it $F$-finite} if $R$ is a finitely generated as an $R$-module via the homomorphism $F$. $R$ is called {\it $F$-pure} if the Frobenius endomorphism is pure.\footnote{A map of $R$-modules $N\rightarrow N'$ is {\it pure} if for every $R$-module $M$ the map $N\otimes_RM\rightarrow N'\otimes_RM$ is injective for every $R$-module $M$.} It is worth to note that if $R$ is either $F$-finite or complete, then $R$ being $F$-pure is equivalent to the condition that the Frobenius endomorphism $F: R \to R$ is split \cite{HR76}. Let $I = (x_1, \ldots, x_t)$ be an ideal of $R$. Then we denote by $H^i_I(R)$ the \textit{i-th local cohomology module} with support at $I$ (we refer to \cite{BS98} for the general theory of local cohomology modules). Recall that local cohomology may be computed as the cohomology of the \v{C}ech complex
$$
0 \to R \to \oplus_{i=1}^t R_{x_i} \to \cdots \to\oplus_{i=1}^tR_{x_1\cdots \widehat{x}_i\cdots x_t}\to R_{x_1 \cdots x_t} \to 0.
$$
The Frobenius endomorphism $F:R \to R$ induces a natural Frobenius action $F:H^i_I(R) \to H^i_{I^{[p]}}(R) \cong H^i_{I}(R)$. A local ring $(R,\fm)$ is called \textit{$F$-injective} if the Frobenius action on $H^i_{\fm}(R)$ is injective for all $i \ge 0$. This is the case if $R$ is $F$-pure \cite[Lemma 2.2]{HR76}. One can also characterize $F$-injectivity using certain ideal closure operations (see \cite{M15, QS16} for more details).


\begin{example}
Let $I = (x_1, \ldots, x_t)\subseteq R$ be an ideal generated by $t$ elements. By the above discussion we have $H^t_I(R) \cong R_{x_1 \cdots x_t}/\im(\oplus_{i=1}^tR_{x_1\cdots \widehat{x}_i\cdots x_t}\to R_{x_1\cdots x_t})$ and the natural Frobenius action on $H^t_I(R)$ sends $\frac{1}{x_1\cdots x_t}$ to $\frac{1}{x_1^p\cdots x_t^p}$. Therefore it is easy to see the Frobenius action on $H^t_I(R)$ is full (in fact, $\mathscr{F}_R(H^t_I(R))\to H^t_I(R)$ is an isomorphism). On the other hand, one cannot expect $H^t_I(R)$ is always anti-nilpotent even when $R$ is regular. For example, let $R = k[[x, y]]$ be a formal power series ring in two variables and $I = (x)$. We have $$H^1_{(x)}(R) \cong k[[y]]x^{-1} \oplus \cdots \oplus k[[y]] x^{-n} \oplus \cdots .$$
Let $N$ be the submodule of $H^1_{(x)}(R)$ generated by $\{y^2x^{-n}\}_{n=1}^\infty$, then it is easy to see $N$ is an $F$-stable submodule of $H^1_{(x)}(R)$. However $F(y x^{-1}) = y^p x^{-p}\in N$ while $yx^{-1}\notin N$. So the Frobenius action on $H^1_{(x)}(R)/N$ is not injective and hence $H^1_{(x)}(R)$ is not anti-nilpotent.
\end{example}

We will be mostly interested in the Frobenius actions on local cohomology modules of $R$ supported at the maximal ideal. We introduce two notions of $F$-singularities.

\begin{definition}
\begin{enumerate}
\item We say that $(R,\fm)$ is $F$-full, if the Frobenius action on $H_{\fm}^i(R)$ is full for every $i\geq 0$. This means $\mathscr{F}_R(H_{\fm}^i(R))\to H_{\fm}^i(R)$ is surjective for every $i\ge 0$.
\item We say that $(R,\fm)$ is $F$-anti-nilpotent, if the Frobenius action on $H^i_{\fm}(R)$ is anti-nilpotent for every $i \geq 0$.

\end{enumerate}
\end{definition}

The concept of $F$-anti-nilpotency is not new, it was introduced and studied in \cite{EH08} and \cite{M14} under the name {\it stably FH-finite}: that is, all local cohomology modules of $R$ and $R[[x_1,\dots,x_n]]$ supported at their maximal ideals have only finitely many $F$-stable submodules. It is a nontrivial result \cite[Theorem 4.15]{EH08} that this is equivalent to $R$ being $F$-anti-nilpotent. 

\begin{remark}
\begin{enumerate}

\item It is clear that $F$-anti-nilpotent implies $F$-injective and $F$-full (see Lemma \ref{anti-nilp implies full}). Moreover, $F$-pure local rings are $F$-anti-nilpotent \cite[Theorem 1.1]{M14}. In particular, $F$-pure local rings are $F$-full.

\item We can construct many $F$-anti-nilpotent (equivalently, stably FH-finite) rings that are not $F$-pure \cite[Sections 5 and 6]{QS16}.

\item Cohen-Macaulay rings are automatically $F$-full, since $\mathscr{F}_R(H_{\fm}^d(R))\to H_{\fm}^d(R)$ is an isomorphism. But even $F$-injective Cohen-Macaulay rings are not necessarily $F$-anti-nilpotent \cite[Example 2.16]{EH08}.


\end{enumerate}
\end{remark}

We give some simple examples of rings that are not $F$-full, we will see a family of such rings in Example \ref{family of nonfull}.

\begin{example}
\begin{enumerate}
\item Let $R=k[s^4, s^3t, st^3, t^4]$ where $k$ is a field of characteristic $p>0$. Then $R$ is a graded ring with $s^4, t^4$ a homogeneous system of parameters. A simple computation shows that the class $$[\frac{(s^3t)^2}{s^4}, -\frac{(st^3)^2}{t^4}]\in R_{s^4}\oplus R_{t^4}$$ spans the local cohomology module $H_{\fm}^1(R)$. In particular, $[H_{\fm}^1(R)]$ sits only in degree 2 and thus the natural Frobenius map kills $H_{\fm}^1(R)$. $R$ is not $F$-full.

\item Let $R=\frac{k[x,y,z]}{x^3+y^3+z^3}\# k[s,t]$ be the Segre product of $A=\frac{k[x,y,z]}{x^3+y^3+z^3}$ and $B=k[s,t]$, where $k$ is a field of characteristic $p>0$ with $p\equiv 2$ mod $3$. Then $R$ is a normal domain, since it is a direct summand of $A\otimes_k B=A[s,t]$. Moreover, a direct computation (for example see \cite[Example 4.11 and 4.16]{MSS16}) shows that $$H_{\fm_R}^2(R)=[H_{\fm_R}^2(R)]_0\cong [H_{\fm_A}^2(A)]_0=k.$$ Since $p\equiv 2$ mod 3, we know the natural Frobenius map kills $[H_{\fm_A}^2(A)]_0$. Hence $R$ is not $F$-full. On the other hand, if $p\equiv 1$ mod $3$, then it is well known that $R$ is $F$-pure (since $A$ is) and hence $F$-anti-nilpotent \cite[Theorem 1.1]{M14}.
\end{enumerate}
\end{example}

\begin{remark}\label{DB analog}
\begin{enumerate}
\item When $R$ is a homomorphic image of a regular ring $A$, say $R=A/I$, $R$ is $F$-full if and only if $H_{\fm}^i(A/J)\to H_{\fm}^i(A/I)$ is surjective for every $J\subseteq I\subseteq \sqrt{J}$. This is because by \cite[Lemma 2.2]{L06}, the $R$-span of $F^e(H_{\fm}^i(R))$ is the same as the image $H_{\fm}^i(A/I^{[p^e]})\to H_{\fm}^i(A/I)$, and for every $J\subseteq I\subseteq\sqrt{J}$, $I^{[p^e]}\subseteq J$ for $e\gg0$. As an application, when $R=A/I$ is $F$-full, we have $H_{\fm}^i(A/I) = 0$ provided $H_{\fm}^i(A/J) = 0$. Hence $\depth A/I \ge \depth A/J $ for every $J\subseteq I\subseteq\sqrt{J}$.

\item Suppose $R$ is a local ring essentially of finite type over $\mathbb{C}$ and $R$ is Du Bois (we refer to \cite{Sch09} or \cite{MSS16} for the definition and basic properties of Du Bois singularities). In this case we do have $H_{\fm}^i(A/J)\to H_{\fm}^i(A/I)$ is surjective for every $J\subseteq I=\sqrt{J}$ \cite[Lemma 3.3]{MSS16}. This is the main ingredient in proving singularities of dense $F$-injective type deform \cite[Theorem C]{MSS16}.

\item Since $F$-injective singularity is the conjectured characteristic $p>0$ analog of Du Bois singularity \cite{Sch09} \cite{BST16}, it is thus quite natural to ask whether $F$-injective local rings are always $F$-full? It turns out that this is false in general \cite[Example 3.5]{MSS16}. However, constructing such examples seems hard. In fact, \cite[Example 2.16]{EH08} (or its variants like \cite[Example 3.5]{MSS16}) is the only example we know that is $F$-injective but not $F$-anti-nilpotent.
\end{enumerate}
\end{remark}

The above remarks motivate us to introduce and study $F$-fullness and a stronger notion of $F$-injectivity (see section 5).

We end this subsection by proving that $F$-full rings localize. Note that it is proved in \cite[Theorem 5.10]{M14} that $F$-anti-nilpotent rings localize.

For convenience, we use $R^{(1)}$ to denote the target ring of the Frobenius map $R\overset{F}{\to} R^{(1)}$. If $M$ is an $R$-module, then $\Hom_R(R^{(1)}, M)$ has a structure of an $R^{(1)}$-module. We can then identify $R^{(1)}$ with $R$, and $\Hom_R(R^{(1)}, M)$ corresponds to an $R$-module which we call $F^\flat(M)$ (we refer to \cite[section 2.3]{BB11} for more details on this). When $R$ is $F$-finite, we have $\Hom_R(R^{(1)}, E_R)\cong E_{R^{(1)}}$ and $F^\flat(E)\cong E_R$, where $E_R$ denotes the injective hull of the residue field of $(R,\fm)$.

\begin{proposition}
Let $(R,\fm)$ be an $F$-finite and $F$-full local ring. Then $R_{\fp}$ is also $F$-full for every $\fp\in\Spec R$.
\end{proposition}
\begin{proof}
By a result of Gabber \cite[Remark 13.6]{G04}, $R$ is a homomorphic image of a regular ring $A$. Let $n=\dim A$. We have
\begin{eqnarray*}
&& \Hom_{R^{(1)}}(\Hom_R(R^{(1)}, \Ext_A^{n-i}(R, A)), E_{R^{(1)}})\\
&\cong& \Hom_{R^{(1)}}(\Hom_R(R^{(1)}, \Ext_A^{n-i}(R, A)),  \Hom_R(R^{(1)}, E_R))\\
&\cong& \Hom_R(\Hom_R(R^{(1)}, \Ext_A^{n-i}(R, A)), E_R)\\
&\cong& R^{(1)}\otimes \Hom_R(\Ext_A^{n-i}(R, A)), E_R)\\
&\cong& R^{(1)}\otimes H_{\fm}^i(R)
\end{eqnarray*}
where the last isomorphism is by local duality. Thus after identifying $R^{(1)}$ with $R$, we have $\mathscr{F}_R(H_{\fm}^i(R))$ is the Matlis dual of $F^\flat(\Ext_A^{n-i}(R, A))$. So $\mathscr{F}_R(H_{\fm}^i(R))\to H^i_{\fm}(R)$ is surjective for every $i$ if and only if $\Ext_A^{n-i}(R, A)\to F^\flat(\Ext_A^{n-i}(R, A))$ is injective for every $i$. The latter condition clearly localizes. So $R$ is $F$-full implies $R_{\fp}$ is $F$-full for every $\fp\in\Spec R$.
\end{proof}

\section{On surjective elements}

The following definition was introduced in \cite{HMS14} and was the key tool in \cite{HMS14}.

\begin{definition}Let $(R, \frak m)$ be a Noetherian local ring and $x$ a regular element of $R$. \rm $x$ is called a {\it surjective element} if the natural map on the local cohomology module $H^i_{\frak m}(R/(x^n)) \to H^i_{\frak m}(R/(x))$ induced by $R/(x^n) \to R/(x)$ is surjective for all $n>0$ and $i \ge 0$.
\end{definition}

The next proposition is a restatement of \cite[Lemma 3.2]{HMS14}, so we omit the proof.

\begin{proposition}\label{L2} The following are equivalent:
\begin{enumerate}[{(i)}]\rm
\item {\it $x$ is a surjective element.}
\item {\it For all $0 < h \le k$ the multiplication map
$$R/(x^h) \overset{x^{k-h}}{\to}R/(x^k)$$
induces an injection
$$H^i_{\frak m}(R/(x^h)) \to H^i_{\frak m}(R/(x^k))$$
for each $i \ge 0$.
}
\item {\it For all $0 < h \le k$ the short exact sequence
$$0 \to R/(x^h) \overset{x^{k-h}}{\to}R/(x^k) \to R/(x^{k-h}) \to 0$$
induces a short exact sequence
$$0 \to H^i_{\frak m}(R/(x^h)) \to H^i_{\frak m}(R/(x^k)) \to H^i_{\frak m}(R/(x^{k-h})) \to 0$$
for each $i \ge 0$.
}
\end{enumerate}
\end{proposition}

\begin{proposition}\label{surj} The following are equivalent:
\begin{enumerate}[{(i)}]\rm
\item {\it $x$ is a surjective element.}
\item {\it The multiplication map $H^i_{\frak m}(R) \overset{x}{\to} H^i_{\frak m}(R)$ is surjective for all $i \ge 0$.}
\end{enumerate}
\end{proposition}
\begin{proof}
By Proposition \ref{L2}, $x$ is a surjective element if and only if all maps in the direct limit system $\{ H^i_{\frak m}(R/(x^h)) \}_{h \ge 1}$ are injective. This is equivalent to the condition
$$\phi_h : H^i_{\frak m}(R/(x^h)) \to \varinjlim_h H^i_{\frak m}(R/(x^h)) \cong H^i_{\frak m}(H^1_{(x)}(R)) \cong H^{i+1}_{\frak m}(R) $$
is injective for all $h \ge 1$ and all $i \ge 0$ (the last isomorphism comes from an easy computation using local cohomology spectral sequences and noting that $x$ is a nonzerodivisor on $R$, see also \cite[Lemma 2.2]{HMS14}).
\begin{claim}
$\phi_h$ is exactly the connection maps in the long exact sequence of local cohomology induced by $0\to R\xrightarrow{\cdot x^h} R\to R/(x^h)\to 0$:
\[
\cdots \to H^i_{\frak m}(R/(x^h)) \overset{\phi_h}{\to} H^{i+1}_{\frak m}(R) \overset{x^h}{\to} H^{i+1}_{\frak m}(R) \to \cdots
 \]
\end{claim}
\begin{proof}[Proof of claim]
Observe that by definition, $\phi_h$ is the natural map in the long exact sequence of local cohomology
$$\cdots\to H_{\fm}^i(R/(x^h))\xrightarrow{\phi_h} H_{\fm}^i(R_x/R) \xrightarrow{\cdot x} H_{\fm}^i(R_x/R)\to\cdots$$
which is induced by $0\to R/(x^h)\to R_x/R \xrightarrow{\cdot x^h} R_x/R \to 0$ (note that $x^h$ is a nonzerodivisor on $R$ and $H_x^1(R)\cong R_x/R$). However, it is easy to see that the multiplication by $x^h$ map $H_{\fm}^i(R_x/R) \xrightarrow{\cdot x^h} H_{\fm}^i(R_x/R)$ can be identified with the multiplication by $x^h$ map $H_{\fm}^{i+1}(R)\xrightarrow{\cdot x^h} H_{\fm}^{i+1}(R)$ because we have a natural identification $H_{\fm}^i(R_x/R)\cong H_{\fm}^i(H_x^1(R))\cong H_{\fm}^{i+1}(R)$ (see for example \cite[Lemma 2.2]{HMS14}). This finishes the proof of the claim.
\end{proof}
From the claim it is immediate that $x$ is a surjective element if and only if the long exact sequence splits into short exact sequences:
$$0 \to H^i_{\frak m}(R/(x^h)) \to H^{i+1}_{\frak m}(R) \overset{x^h}{\to} H^{i+1}_{\frak m}(R) \to 0.$$
But this is equivalent to saying that the multiplication map $H^i_{\frak m}(R) \overset{x^h}{\to} H^i_{\frak m}(R)$ is surjective for all $h\ge 1$ and $i \ge 0$, and also equivalent to $H^i_{\frak m}(R) \overset{x}{\to} H^i_{\frak m}(R)$ is surjective for all $i\geq 0$.
\end{proof}

We next link the notion of surjective element with $F$-fullness. This is inspired by \cite{MSS16, SW07}.

\begin{proposition}
\label{full imp surj}
Let $x$ be a regular element of $(R,\fm)$. If $R/(x)$ is $F$-full, then $x$ is a surjective element. In particular, if $R/(x)$ is $F$-anti-nilpotent, then $x$ is a surjective element.
\end{proposition}
\begin{proof}
We have natural maps: $$\mathscr{F}^e_R(H_{\fm}^i(R/(x)))\xrightarrow{\alpha_e} R/(x)\otimes_R\mathscr{F}^e_{R}(H_{\fm}^i(R/(x)))\cong  \mathscr{F}^e_{R/(x)}(H_{\fm}^i(R/(x)))\xrightarrow{\beta_e}H_{\fm}^i(R/(x)).$$
If $R/(x)$ is $F$-full, then $\beta_e$ is surjective for every $e$. Since $\alpha_e$ is always surjective, the natural map $\mathscr{F}^e_R(H_{\fm}^i(R/(x)))\to H_{\fm}^i(R/(x))$ is surjective for every $e$. Now simply notice that for every $e>0$, the map $\mathscr{F}_R^e(H_{\fm}^i(R/(x)))\to H_{\fm}^i(R/(x))$ factors through $H_{\fm}^i(R/(x^{p^e}))\to H_{\fm}^i(R/(x))$, so $H_{\fm}^i(R/(x^{p^e}))\to H_{\fm}^i(R/(x))$ is surjective for every $e>0$. This clearly implies that $x$ is a surjective element.
\end{proof}


The above propositions allow us to construct a family of non $F$-full local rings:

\begin{example}
\label{family of nonfull}
Let $(R,\fm)$ be a local ring with finite length cohomology, i.e., $H_{\fm}^i(R)$ has finite length for every $i<\dim R$ (under mild conditions, this is equivalent to saying that $R$ is Cohen-Macaulay on the punctured spectrum). Let $x$ be an arbitrary regular element in $R$. If $R$ is not Cohen-Macaulay, then we claim that $R/(x)$ is not $F$-full (and hence not $F$-anti-nilpotent). For suppose it is, then $x$ is a surjective element by Proposition \ref{full imp surj}, hence $H_{\fm}^i(R)\overset{x}{\to} H_{\fm}^i(R)$ is surjective for every $i$ by Proposition \ref{surj}. But since $R$ has finite length cohomology, we also know that a power of $x$ annihilates $H_{\fm}^i(R)$ for every $i<\dim R$. This implies $H_{\fm}^i(R)=0$ for every $i<\dim R$. So $R$ is Cohen-Macaulay, a contradiction.
\end{example}

We learned the following argument from \cite[Lemma A.1]{HMS14}. Since it is a crucial technique of this paper, we provide a detailed proof.

\begin{proposition}\label{crutical tech} Let $(R, \fm)$ be a local ring of prime characteristic $p$ and $x$ a regular element of $R$. Let $s$ be a positive integer such that the map $H^{s-1}_{\fm}(R) \overset{x}{\to} H^{s-1}_{\fm}(R)$ is surjective and the Frobenius action on $H^{s-1}_{\fm}(R/(x))$ is injective, then the map
$$H^{s}_{\fm}(R) \overset{x^{p-1}F}{\longrightarrow} H^{s}_{\fm}(R)$$
is injective.
\end{proposition}
\begin{proof}The natural commutative diagram
$$
\begin{CD}
0 @>>>  R  @>x>> R @>>> R/(x)  @>>> 0\\
@. @Vx^{p-1}FVV @VFVV @VFVV \\
0 @>>>  R  @>x>> R @>>> R/(x) @>>> 0
\end{CD}
$$
induces the following commutative diagram (the left most $0$ comes from our hypothesis that the map $H^{s-1}_{\fm}(R) \overset{x}{\to} H^{s-1}_{\fm}(R)$ is surjective):
$$
\begin{CD}
0 @>>>  H^{s-1}_{\fm}(R/(x))  @>\alpha>> H^s_{\fm}(R) @>x>> H^s_{\fm}(R)  @>>> \cdots\\
@. @VFVV @Vx^{p-1}FVV @VFVV\\
0 @>>>  H^{s-1}_{\fm}(R/(x))  @>\alpha>> H^s_{\fm}(R) @>x>> H^s_{\fm}(R)  @>>> \cdots.
\end{CD}
$$
Suppose $y \in \ker (x^{p-1}F) \cap \mathrm{Soc}(H^s_{\fm}(R))$. Then we have $x\cdot y = 0$ so there exists $z \in H^{s-1}_{\fm}(R/(x))$ such that $\alpha(z) = y$. Following the above commutative diagram we have
$$(\alpha \circ F) (z) = x^{p-1}F(\alpha (z)) = x^{p-1}F(y)=0.$$
However, since both $F$ and $\alpha$ are injective, we have $z = 0$ and hence $y=0$. This shows $x^{p-1}F$ is injective and hence completes the proof.
\end{proof}

Proposition \ref{crutical tech} immediately generalizes the main result of \cite{HMS14}:

\begin{corollary}[compare with \cite{HMS14}, Main Theorem] \label{sur imp injectivity}
Let $(R, \fm)$ be a local ring of prime characteristic $p$ and $x$ a regular element of $R$. Suppose $R/(x)$ is $F$-injective. Then we have
\begin{enumerate}[{(i)}]\rm
\item {\it The map $H^{t}_{\fm}(R) \overset{x^{p-1}F}{\longrightarrow} H^{t}_{\fm}(R)$ is injective where $t=\depth R$. In particular, the natural Frobenius action on $H^t_{\fm}(R)$ is injective.}
\item {\it Suppose $x$ is a surjective element. Then the map $H^{i}_{\fm}(R) \overset{x^{p-1}F}{\longrightarrow} H^{i}_{\fm}(R)$ is injective for all $i\ge 0$. In particular, $R$ is $F$-injective.}
\item {\it If $R/(x)$ is $F$-full (e.g., $R$ is $F$-anti-nilpotent or $R$ is $F$-pure), then $R$ is $F$-injective.}
\end{enumerate}
\end{corollary}
\begin{proof}
(i) follows from Proposition \ref{crutical tech} applied to $s=t$. (ii) also follows from Proposition \ref{crutical tech} (because $H^{i}_{\fm}(R) \overset{x}{\to} H^{i}_{\fm}(R)$ is surjective for every $i\geq 0$ by Proposition \ref{surj}). (iii) follows from (ii), because we know $x$ is a surjective element by Proposition \ref{full imp surj}.
\end{proof}

In the next two sections, we will show that $F$-full and $F$-anti-nilpotent singularities both deform. We will also prove new cases of deformation of $F$-injectivity. These results are generalizations of Proposition \ref{crutical tech} and Corollary \ref{sur imp injectivity}.

\section{Deformation of $F$-full and $F$-anti-nilpotent singularities}
In this section we prove that the condition $F$-full and $F$-anti-nilpotent both deform. Throughout this section we assume that $(R, \fm)$ is a local ring of prime characteristic $p$. We begin with a crucial lemma.

\begin{lemma}\label{crucial lem}
Let $x$ be a surjective element of $R$. Let $N\subseteq H_{\fm}^i(R)$ be an $F$-stable submodule. Let $L=\cap_t x^tN$. Then $L$ is an $F$-stable submodule of $H_{\fm}^i(R)$ and we have the following commutative diagram (for every $e\geq 1$):
$$
\begin{CD}
0 @>>>  H^{i-1}_{\fm}(R/(x))/\phi^{-1}(L)  @>\phi>> H^i_{\fm}(R)/L @>x>> H^i_{\fm}(R)/L @>>> 0\\
@. @VF^eVV @Vx^{p^e-1}F^eVV @VF^eVV \\
0 @>>>  H^{i-1}_{\fm}(R/(x))/\phi^{-1}(L)  @>\phi>> H^i_{\fm}(R)/L @>x>> H^i_{\fm}(R)/L @>>> 0,
\end{CD}
$$
where $\phi$ is the map $H^{i-1}_{\fm}(R/(x))\to H^{i}_{\fm}(R)$.
\end{lemma}
\begin{proof}
Since $x$ is a surjective element, by Proposition \ref{surj} we know that the map $$H^i_{\fm}(R) \overset{x}{\to} H^i_{\fm}(R) \text{ is surjective for every } i>0. \quad (\star)$$ Applying the local cohomology functor to the following commutative diagram:
$$
\begin{CD}
0 @>>>  R  @>x>> R @>>> R/(x)  @>>> 0\\
@. @Vx^{p^e-1}F^eVV @VF^eVV @VF^eVV \\
0 @>>>  R  @>x>> R @>>> R/(x) @>>> 0,
\end{CD}
$$
we have the following commutative diagram:
$$
\begin{CD}
0 @>>>  H^{i-1}_{\fm}(R/(x)) @>\phi>> H^i_{\fm}(R) @>x>> H^i_{\fm}(R) @>>> 0\\
@. @VF^eVV @Vx^{p^e-1}F^eVV @VF^eVV \\
0 @>>>  H^{i-1}_{\fm}(R/(x))  @>\phi>> H^i_{\fm}(R)@>x>> H^i_{\fm}(R) @>>> 0
\end{CD}
$$
for all $i \ge 1$ and $e\geq 1$, where the rows are short exact sequences by $(\star)$.

Therefore to prove the lemma, it suffices to show that $L$ is $F$-stable and
$$0 \to H^{i-1}_{\fm}(R/(x))/\phi^{-1}(L) \overset{\phi}{\to} H^i_{\fm}(R)/L \overset{x}{\to} H^i_{\fm}(R)/L \to 0$$ is exact. It is clear that $L$ is $F$-stable since it is an intersection of $F$-stable submodules of $H_{\fm}^i(R)$. To see the exactness of the above sequence, first note that $\im(\phi)=0:_{H_{\fm}(R)}x$, so $L+\im(\phi)\subseteq L:_{H_{\fm}(R)}x$. Thus it is enough to check that $L:_{H_{\fm}(R)}x \subseteq L + \im (\phi)$. Let $y$ be an element such that $xy \in L$. Since $L = xL$ by the construction of $L$, there exists $z \in L$ such that $xy = xz$. So $y-z \in \im(\phi)$ and hence $y \in L + \im (\phi)$, as desired.
\end{proof}


We are ready to prove the main result of this section. This answers \cite[Problem 4]{QS16} for stably FH-finiteness.
\begin{theorem}\label{stably}
$(R, \fm)$ be a local ring of positive characteristic $p$ and $x$ a regular element of $R$. Then we have:
\begin{enumerate}[{(i)}]\rm
\item If $R/(x)$ is $F$-anti-nilpotent, then so is $R$.
\item If $R/(x)$ is $F$-full, then so is $R$.
\end{enumerate}
\end{theorem}
\begin{proof}
We first prove (i). Let $N$ be an $F$-stable submodule of $H^i_{\fm}(R)$. We want to show that the induced Frobenius action on $H^i_{\fm}(R)/N$ is injective. Since $R/(x)$ is $F$-anti-nilpotent, $x$ is a surjective element by Proposition \ref{full imp surj}. Let $L = \cap_t x^tN$. By Lemma \ref{crucial lem}, we have the following commutative diagram:
$$
\begin{CD}
0 @>>>  H^{i-1}_{\fm}(R/(x))/\phi^{-1}(L)  @>\phi>> H^i_{\fm}(R)/L @>x>> H^i_{\fm}(R)/L @>>> 0\\
@. @VF^eVV @Vx^{p^e-1}F^eVV @VF^eVV \\
0 @>>>  H^{i-1}_{\fm}(R/(x))/\phi^{-1}(L)  @>\phi>> H^i_{\fm}(R)/L @>x>> H^i_{\fm}(R)/L @>>> 0.
\end{CD}
$$

We first claim that the middle map $x^{p^e-1}F^e: H^i_{\fm}(R)/L \to H^i_{\fm}(R)/L$ is injective. Let $y \in \ker(x^{p^e-1}F^e) \cap \mathrm{Soc}(H^i_{\fm}(R)/L)$. We have $x\cdot y = 0$, so $y = \phi(z)$ for some $z \in H^{i-1}_{\fm}(R/(x))/\phi^{-1}(L)$. It is easy to see that $\phi^{-1}(L)$ is an $F$-stable submodule of $H^{i-1}_{\fm}(R/(x))$ and $F^e(z)=0$. Since $R/(x)$ is $F$-anti-nilpotent, we know the Frobenius action $F$, and hence its iterate $F^e$, on $H^{i-1}_{\fm}(R/(x))/\phi^{-1}(L)$ is injective. Therefore, $z=0$ and hence $y=0$. This proves that $x^{p^e-1}F^e$ and hence $F$ acts injectively on $H^i_{\fm}(R)/L$.

Note that we have a descending chain $ N\supseteq xN\supseteq x^2N\supseteq\cdots.$ Since $H_{\fm}^i(R)$ is Artinian, $L = \cap_tx^tN=x^nN$ for all $n\gg 0$. We next claim that $L =N$, this will finish the proof because we already showed $F$ acts injectively on $H^i_{\fm}(R)/L$. We have $x^{p^e-1}F^e(N)\subseteq x^{p^e-1}N=L$ for $e\gg 0$, but the map $x^{p^e-1}F^e: H^i_{\fm}(R)/L \to H^i_{\fm}(R)/L$ is injective by the above paragraph. So we must have $N\subseteq L$ and thus $L = N$. This completes the proof of (1).

Next we prove (ii). The method is similar to that of (i). Let $N$ be the $R$-span of $F(H_{\fm}^i(R))$ in $H_{\fm}^i(R)$, this is  the same as the image of $\mathscr{F}_R(H_{\fm}^i(R))\to H_{\fm}^i(R)$. It is clear that $N$ is an $F$-stable submodule. We want to show $N=H_{\fm}^i(R)$. Since $R/(x)$ is $F$-full, $x$ is a surjective element by Proposition \ref{full imp surj}. Let $L = \cap_t x^tN$. By Lemma \ref{crucial lem}, we have the following commutative diagram:
$$
\begin{CD}
0 @>>>  H^{i-1}_{\fm}(R/(x))/\phi^{-1}(L)  @>\phi>> H^i_{\fm}(R)/L @>x>> H^i_{\fm}(R)/L @>>> 0\\
@. @VF^eVV @Vx^{p^e-1}F^eVV @VF^eVV \\
0 @>>>  H^{i-1}_{\fm}(R/(x))/\phi^{-1}(L)  @>\phi>> H^i_{\fm}(R)/L @>x>> H^i_{\fm}(R)/L @>>> 0.
\end{CD}
$$

The descending chain $N\supseteq xN\supseteq x^2N\supseteq\cdots$ stabilizes because $H_{\fm}^i(R)$ is Artinian. So $L=\cap_t x^tN=x^nN$ for $n\gg 0$. The key point is that in the above diagram, the middle Frobenius action  $x^{p^e-1}F^{e}$ is the zero map on $H_{\fm}^i(R)/L$ for $e\gg 0 $, because for any $y\in H_{\fm}^i(R)$, $F^e(y)\in N$ and thus $x^{p^e-1}F^{e}(y)\in L$ for $e\gg 0$. But then since $H_{\fm}^{i-1}(R/(x))/{\phi^{-1}(L)}$ can be viewed as a submodule of $H_{\fm}^i(R)/L$ by the above commutative diagram, the natural Frobenius action $F^e$ on $H_{\fm}^{i-1}(R/(x))/{\phi^{-1}(L)}$ is zero, i.e., $F$ is nilpotent on $H_{\fm}^{i-1}(R/(x))/{\phi^{-1}(L)}$.

Since $F$ is nilpotent on $H_{\fm}^{i-1}(R/(x))/{\phi^{-1}(L)}$, we know that $\phi^{-1}(L)$ must contain all elements $F^e(H_{\fm}^i(R/(x)))$, hence it contains the $R$-span of $F^e(H_{\fm}^i(R/(x)))$. But $R/(x)$ is $F$-full, so we must have $\phi^{-1}(L)=H_{\fm}^{i-1}(R/(x))$. But this means the map $$H_{\fm}^i(R)/L\xrightarrow{x}H_{\fm}^i(R)/L$$ is an isomorphism, which is impossible unless $H_{\fm}^i(R)=L$ (since otherwise any nonzero socle element of $H_{\fm}^i(R)/L$ maps to zero). Therefore we have $H_{\fm}^i(R)=N=L$. This proves $R$ is $F$-full and hence finished the proof of (2).
\end{proof}

The following is a well-known counter-example of Fedder \cite{F83} and Singh \cite{Sin99a} for the deformation of $F$-purity.
\begin{example}[compare with \cite{QS16}, Lemma 6.1] Let $K$ be a perfect field of characteristic $p>0$ and let
$$
R:=K[[U,V,Y,Z]]/(UV,UZ,Z(V-Y^2)).
$$
Let $u,v,y$ and $z$ denote the image of $U, V, Y$ and $Z$ in $R$ (and its quotients), respectively. Then $y$ is a regular element of $R$ and $R/(y) \cong K[[U, V, Z]]/(UV, UZ, VZ)$ is $F$-pure by \cite[Proposition 5.38]{HR76}. So $R/(y)$ is $F$-anti-nilpotent  by \cite[Theorem 1.1]{M14}. By Theorem \ref{stably} we have $R$ is also $F$-anti-nilpotent, or equivalently, $R$ is stably $FH$-finite.
\end{example}


\section{$F$-injectivity}
\subsection{$F$-injectivity and depth}
We start with the following definition.
\begin{definition}(cf. \cite[Definition 9.1.3]{BS98})
\label{nCM3}
Let $M$ be a finitely generated module over a local ring $(R,\fm)$. The \textit{finiteness dimension $f_{\fm}(M)$ of $M$ with respect to} $\fm$ is defined as follows:
$$
f_{\fm}(M):=\mathrm{inf} \{i \,|\, H^i_{\fm}(M)~\text{is not finitely generated} \} \in \mathbb{Z}_{\ge 0} \cup \{\infty\}.
$$
\end{definition}

\begin{remark}\label{finite dim}
\begin{enumerate}[{(i)}]
\item Assume that $\dim M=0$ or $M=0$ (recall that a trivial module has dimension $-1$). In this case, $H^i_{\fm}(M)$ is finitely generated for all $i$ and $f_{\fm}(M)$ is equal to $\infty$. It will be essential to know when the finiteness dimension is a positive integer. We mention the following result. Let $(R,\fm)$ be a local ring and let $M$ be a finitely generated $R$-module. If $d=\dim M >0$, then the local cohomology module $H^d_{\fm}(M)$ is not finitely generated. For the proof of this result, see \cite[Corollary 7.3.3]{BS98}.
\item Suppose $(R, \fm)$ is an image of a Cohen-Macaulay local ring. By the Grothendieck finiteness theorem (cf. \cite[Theorem 9.5.2]{BS98}) we have
$$f_{\fm}(M) = \min \{\depth M_{\fp} + \dim R/\fp \, : \, \fp \in \mathrm{Supp}(M) \setminus \{\fm \} \}.$$
\item $M$ is {\it generalized Cohen-Macaulay} if and only if $\dim M = f_{\fm}(M)$.
\end{enumerate}
\end{remark}

It is clear that $\depth R \le f_{\fm}(R) \le \dim R$. The following result says that if $R/(x)$ is $F$-injective, then $R$ has ``good" depth.
\begin{theorem} \label{googdepth}
If $R/(x)$ is $F$-injective, then $\depth R = f_{\fm}(R)$.
\end{theorem}
\begin{proof} Suppose $t = \depth R < f_{\fm}(R)$. The commutative diagram
$$
\begin{CD}
0 @>>>  R  @>x>> R @>>> R/(x)  @>>> 0\\
@. @Vx^{p-1}FVV @VFVV @VFVV \\
0 @>>>  R  @>x>> R @>>> R/(x) @>>> 0
\end{CD}
$$
induces the following commutative diagram
$$
\begin{CD}
0 @>>>  H^{t-1}_{\fm}(R/(x))  @>\alpha>> H^t_{\fm}(R)  @>>> \cdots\\
@. @VF^eVV @V x^{p^e-1}F^eVV \\
0 @>>>  H^{t-1}_{\fm}(R/(x))  @>\alpha>> H^t_{\fm}(R)  @>>> \cdots,
\end{CD}
$$
where both $\alpha$ and the left vertical map are injective. But $H^t_{\fm}(R)$ has finite length, $x^{p^e-1}F^e: H^t_{\fm}(R) \to H^t_{\fm}(R)$ vanishes for $e \gg 0$, which is a contradiction.
\end{proof}

\begin{remark}
The assertion of Theorem \ref{googdepth} also holds true if $R/(x)$ is $F$-full. Indeed, by Proposition \ref{full imp surj} we have $x$ is a surjective element. Hence there is no nonzero $H^i_{\fm}(R)$ of finite length. Thus $\depth R = f_{\fm}(R)$.
\end{remark}

\begin{remark} The above result is closely related to the work of Schwede and Singh in \cite[Appendix]{HMS14}. In the proof of \cite[Lemma A.2, Theorem A.3]{HMS14}, it is claimed that if $R_{\fp}$ satisfies the Serre condition $(S_k)$ for all $\fp$ in $\mathrm{Spec}^{\circ}(R)$, the punctured spectrum of $R$, and $\depth R = t < k$, then $H^t_{\fm}(R)$ is finitely generated. But this fact may not be true if $R$ is not equidimensional. For instance, let $R = K[[a,b,c,d]]/(a) \cap (b,c,d)$ with $K$ a field. We have $\depth R = 1$ and $R_{\fp}$ satisfies $(S_2)$ for all $\fp \in \mathrm{Spec}^{\circ}(R)$. However, $H^1_{\fm}(R)$ is not finitely generated.
\end{remark}

The assertion of \cite[Lemma A.2]{HMS14} (and hence \cite[Theorem A.3]{HMS14}) is still true. In fact, we can reduce it to the case that $R$ is equidimensional. We fill this gap below.

\begin{corollary}[\cite{HMS14}, Lemma A.2] Let $(R, \fm)$ be an $F$-finite local ring. Suppose there exists a regular element $x$ such that $R/(x)$ is $F$-injective. If $R_{\fp}$ satisfies the Serre condition $(S_k)$ for all $\fp \in \mathrm{Spec}^{\circ}(R)$, then $R$ is $(S_k)$.
\end{corollary}
\begin{proof} We can assume that $k \le d = \dim R$. In fact, we need only to prove that $t: = \depth R \ge k$. The case $k = 1$ is trivial since $R$ contains a regular element $x$. For $k \ge 2$, since $R/(x)$ is $F$-injective we have $R/(x)$ is reduced (cf. \cite[Proposition 4.3]{Sch09}). Hence $\depth (R/(x)) \ge 1$, so $\depth R \ge 2$. Thus $R$ satisfies the Serre condition $(S_2)$. On the other hand, since $R$ is $F$-finite, $R$ is a homomorphic image of a regular ring by a result of Gabber \cite[Remark 13.6]{G04}. In particular, $R$ is universally catenary.\footnote{Another way to see this is to use the fact that $F$-finite rings are excellent \cite{K76} and hence universally catenary.} But if a universally catenary ring satisfies $(S_2)$, then it is equidimensional (see \cite[Remark 2.2 (h)]{HH94}). By Theorem \ref{googdepth} and Remark \ref{finite dim} (ii), there exists a prime ideal $\fp \in \mathrm{Spec}^{\circ}(R)$ such that $\depth R = \depth R_{\fp}  + \dim R/\fp$. It is then easy to see that $\depth R \ge \min \{d, k+1\} \ge k$. The proof is complete.
\end{proof}

\begin{remark} In the above argument, we actually proved that if $k<d$, then $\depth R \ge k+1$.
\end{remark}

\subsection{Deformation of $F$-injectivity} We begin with the following generalization of the notion of surjective elements.
\begin{definition}[cf. \cite{CMN04}]
A regular element $x$ is called a {\it strictly filter regular} element if
  $$\coker (H^i_{\fm}(R) \overset{x}{\to} H^i_{\fm}(R))$$
has finite length for all $i \ge 0$.
\end{definition}

\begin{lemma}\label{injec quotient} Let $(R,\fm)$ be a local ring of characteristic $p>0$. Suppose the residue field $k=R/\fm$ is perfect. Let $M$ be an $R$-module with an injective Frobenius action $F$. Suppose $L$ is an $F$-stable submodule of $M$ of finite length. Then the induced Frobenius action on $M/L$ is injective.
\end{lemma}
\begin{proof}
First we note that $L$ is killed by $\fm$: suppose $x\in L$, then $F^e(\fm\cdot x)=\fm^{[p^e]}\cdot x=0$ for $e\gg0$ since $L$ has finite length. But then $\fm\cdot x=0$ since $F$ acts injectively. Now we have a Frobenius action $F$ on a $k$-vector space $L$. Call the image of $L'\subseteq L$ (which is a $k^p$-vector subspace of $L$). Since $F$ is injective, the $k^p$-vector space dimension of $L'$ is equal to the $k$-vector space dimension of $L$. But since $k^p=k$, this implies $L'=L$ and thus $F$ is surjective, hence $F$ is bijective. Now by the injectivity of $F$ again we have $F(x) \notin L$ for all $x \notin L$. Thus $F: M/L \to M/L$ is injective.
\end{proof}

\begin{example}

The perfectness of the residue field in Lemma \ref{injec quotient} is necessary.  Let $A = \mathbb{F}_p[t]$ and $R =k= \mathbb{F}_p(t)$, where $t$ is an indeterminate. We consider the Frobenius action on the $A$-module $Ae_1 \oplus Ae_2$ defined by
  $$F(f(t),g(t)) = (f(t)^p + t g(t)^p,0).$$
It is clear that $F$ is injective. Moreover, $Ae_1 \oplus 0$ is an $F$-stable submodule of $Ae_1 \oplus Ae_2$. Since $F(Ae_1 \oplus Ae_2) \subseteq Ae_1 \oplus 0$, the induced Frobenius action on $(Ae_1 \oplus Ae_2)/(Ae_1 \oplus 0)$ is the zero map. By localizing, we obtain an injective Frobenius action on $M = k\cdot e_1 \oplus k \cdot e_2$ with $L = k\cdot e_1 \oplus 0$ is an $F$-stable submodule of finite length, but the induced Frobenius action on $M/L$ is not injective.
\end{example}

The following is a generalization of the main result of \cite{HMS14} when $R/\fm$ is perfect.

\begin{theorem}\label{stri filt}
Let $(R, \fm)$ be a Noetherian local ring of characteristic $p>0$. Suppose the residue field $k = R/\fm$ is perfect. Let $x$ be a strictly filter regular element. If $R/(x)$ is $F$-injective, then the map $x^{p-1}F$: $H_{\fm}^i(R)\to H_{\fm}^i(R)$ is injective for every $i$, in particular $R$ is $F$-injective.
\end{theorem}
\begin{proof}
Let $L_i: = \coker (H^i_{\fm}(R) \overset{x}{\to} H^i_{\fm}(R))$, we have $L_i$ has finite length for all $i \ge 0$. The commutative diagram
$$
\begin{CD}
0 @>>>  R  @>x>> R @>>> R/(x)  @>>> 0\\
@. @Vx^{p-1}FVV @VFVV @VFVV \\
0 @>>>  R  @>x>> R @>>> R/(x) @>>> 0
\end{CD}
$$
induces the following commutative diagram
$$
\begin{CD}
0 @>>> L_{i-1} @>>> H^{i-1}_{\fm}(R/(x))  @>\phi>> H^i_{\fm}(R) @>x>> H^i_{\fm}(R) @>>> \cdots\\
@. @VFVV @VFVV @Vx^{p-1}FVV @VFVV  \\
0 @>>> L_{i-1} @>>>  H^{i-1}_{\fm}(R/(x))  @>\phi>> H^i_{\fm}(R) @>x>> H^i_{\fm}(R) @>>> \cdots.
\end{CD}
$$
Therefore we have the following commutative diagram
$$
\begin{CD}
0 @>>>  H^{i-1}_{\fm}(R/(x))/L_{i-1}  @>\alpha>> H^i_{\fm}(R) @>x>> H^i_{\fm}(R) @>>> \cdots\\
@. @VFVV @Vx^{p-1}FVV @VFVV \\
0 @>>>  H^{i-1}_{\fm}(R/(x))/L_{i-1}  @>\alpha>> H^i_{\fm}(R) @>x>> H^i_{\fm}(R) @>>> \cdots
\end{CD}
$$
with the Frobenius action $F: H^{i-1}_{\fm}(R/(x))/L_{i-1} \to H^{i-1}_{\fm}(R/(x))/L_{i-1}$ is injective by Lemma \ref{injec quotient}. Now by the same method as in the proof of Proposition \ref{crutical tech} or Theorem \ref{stably} (i), we conclude that the map $x^{p-1}F: H^i_{\fm}(R) \to H^i_{\fm}(R)$ is injective for all $i \ge 0$.
\end{proof}

Similarly, we have the following:

\begin{proposition}\label{stric filt 2}
Let $(R, \fm)$ be a Noetherian local ring of characteristic $p>0$. Suppose the residue field $k = R/\fm$ is perfect. Let $x$ be a regular element such that $R/(x)$ is $F$-injective. Let $s$ be a positive integer such that $H^{s-1}_{\fm}(R/(x))$ has finite length. Then the map $x^{p-1}F: H^{s+1}_{\fm}(R)  \to H^{s+1}_{\fm}(R)$ is injective.
\end{proposition}
\begin{proof} The short exact sequence
$$0 \to R \overset{x}{\to} R \to R/(x) \to 0$$
induces the exact sequence
$$\cdots \to H^{s-1}_{\fm}(R/(x)) \to H^s_{\fm}(R) \overset{x}{\to} H^s_{\fm}(R) \to H^{s}_{\fm}(R/(x)) \to H^{s+1}_{\fm}(R)  \to \cdots.$$
Since $H^{s-1}_{\fm}(R/(x))$ has finite length, so is $\ker(H^s_{\fm}(R) \overset{x}{\to} H^s_{\fm}(R))$. We claim that $$L_s: = \coker(H^s_{\fm}(R) \overset{x}{\to} H^s_{\fm}(R))$$ also has finite length: to see this we may assume $R$ is complete,  since $\ker(H^s_{\fm}(R) \overset{x}{\to} H^s_{\fm}(R))$ has finite length, this means $H^s_{\fm}(R)^\vee \overset{x}{\to} H^s_{\fm}(R)^\vee$ is surjective when localizing at any $\fp\neq\fm$. But by \cite[Theorem 2.4]{Mat86} this implies $H^s_{\fm}(R)^\vee \overset{x}{\to} H^s_{\fm}(R)^\vee$ is an isomorphism when localizing at any $\fp\neq\fm$. Thus $\ker(H^s_{\fm}(R)^\vee \overset{x}{\to} H^s_{\fm}(R)^\vee)$ has finite length which, after dualizing, shows that $\coker(H^s_{\fm}(R) \overset{x}{\to} H^s_{\fm}(R))$ has finite length.

We have proved $L_s=\coker(H^s_{\fm}(R) \overset{x}{\to} H^s_{\fm}(R))$ has finite length. Now the map $x^{p-1}F: H^{s+1}_{\fm}(R)  \to H^{s+1}_{\fm}(R)$ is injective by the same argument as in Theorem \ref{stri filt}.
\end{proof}

The following immediate corollary of the above proposition recovers (and in fact generalizes) results in \cite{HMS14}.
\begin{corollary}[\cite{HMS14}, Corollary 4.7]
  Let $(R, \fm)$ be a Noetherian local ring of characteristic $p>0$. Suppose the residue field $k = R/\fm$ is perfect. Let $x$ be a regular element such that $R/(x)$ is $F$-injective. Then the map $x^{p-1}F: H^{i}_{\fm}(R)  \to H^{i}_{\fm}(R)$ is injective for all $i \le f_{\fm}(R/(x))+1$. In particular, if $R/(x)$ is generalized Cohen-Macaulay, then $R$ is $F$-injective.
\end{corollary}

Because of the deep connections between $F$-injective and Du Bois singularities \cite{Sch09, BST16} and Remark \ref{DB analog}, we believe that it is rarely the case that an $F$-injective ring fails to be $F$-full (again, the only example we know this happens is \cite[Example 3.5]{MSS16}, which is based on the construction of \cite[Example 2.16]{EH08}). Therefore we introduce:

\begin{definition}
We say $(R,\fm)$ is {\it strongly $F$-injective} if $R$ is $F$-injective and $F$-full.
\end{definition}

\begin{remark}
In general we have: $F$-anti-nilpotent $\Rightarrow$ strongly $F$-injective $\Rightarrow$ $F$-injective. Moreover, when $R$ is Cohen-Macaulay, strongly $F$-injective is equivalent to $F$-injective.
\end{remark}

We can prove that strong $F$-injectivity deform.

\begin{corollary}\label{str deform}
Let $x$ be a regular element on  $(R,\fm)$. If $R/(x)$ is strongly $F$-injective, then $R$ is strongly $F$-injective.
\end{corollary}
\begin{proof}
We know $R$ is $F$-injective by Corollary \ref{sur imp injectivity} (iii). But we also know $R$ is $F$-full by Theorem \ref{stably} (ii). This shows that $R$ is strongly $F$-injective.
\end{proof}

\end{document}